\documentclass{amsart}
\usepackage{amssymb}

\usepackage{graphicx}
\usepackage[dvips]{color}

\newtheorem{theorem}{Theorem}[section]
\newtheorem{lemma}[theorem]{Lemma}

\newtheorem{proposition}[theorem]{Proposition}
\newtheorem{claim}[theorem]{Claim}
\newtheorem{conjecture}[theorem]{Conjecture}
\newtheorem{corollary}[theorem]{Corollary}

\theoremstyle{definition}

\theoremstyle{remark}
\newtheorem{remark}[theorem]{Remark}

\numberwithin{equation}{section}

\begin{document}

\title[Generalized torsion elements and bi-orderability]{Generalized torsion elements and bi-orderability of 3-manifold groups}



\author[K. Motegi]{Kimihiko Motegi}
\address{Department of Mathematics, Nihon University, 
3-25-40 Sakurajosui, Setagaya-ku, 
Tokyo 156--8550, Japan}
\email{motegi@math.chs.nihon-u.ac.jp}
\thanks{The first named author has been partially supported by JSPS KAKENHI Grant Number JP26400099 and Joint Research Grant of Institute of Natural Sciences at Nihon University for 2016. }
\author[M. Teragaito]{Masakazu Teragaito}
\address{Department of Mathematics and Mathematics Education, Hiroshima University,
1-1-1 Kagamiyama, Higashi-hiroshima 739--8524, Japan.}
\email{teragai@hiroshima-u.ac.jp}
\thanks{The second named author has been partially supported by JSPS KAKENHI Grant Number JP16K05149.}

\subjclass[2010]{Primary 57M25; Secondary 57M05, 06F15, 20F05, 20F60}

\date{}


\commby{}

\begin{abstract}
It is known that a bi-orderable group has no generalized torsion element, 
but the converse does not hold in general.
We conjecture that the converse holds for the fundamental groups of $3$--manifolds,
and verify the conjecture for non-hyperbolic, geometric $3$--manifolds.  
We also confirm the conjecture for some infinite families of closed hyperbolic $3$--manifolds. 
In the course of the proof, 
we prove that each standard generator of the Fibonacci group $F(2, m)$ ($m > 2$) is a generalized torsion element. 
\end{abstract}

\maketitle


\section{Introduction}

A group $G$ is said to be \textit{bi-orderable\/} if $G$ admits
a strict total ordering $<$ which is invariant under the multiplication from
left and right sides.
That is, if $g<h$, then $agb<ahb$ for any $g,h,a,b\in G$.
In this paper, the trivial group $\{1\}$ is considered to be bi-orderable.

Let $g\in G$ be a non-trivial element.
If some non-empty finite product of conjugates of $g$ equals to the identity,
then $g$ is called a \textit{generalized torsion element}.
In particular, any non-trivial torsion element is a generalized torsion element.
If a group $G$ is bi-orderable, then $G$ has no generalized torsion element
(see Lemma~\ref{lem:bo}).
In other words, the existence of generalized torsion element is an obstruction for 
bi-orderability.
In the literature \cite{BL,LMR,MR0,MR}, 
a group without generalized torsion element is called an $R^*$--group or a $\Gamma$--torsion-free group.
Thus bi-orderable groups are $R^*$--groups.
However, the converse does not hold in general \cite[Chapter 4]{MR}. 

If we restrict ourselves to a specific class of groups, say, knot groups 
or more generally, $3$--manifold groups, then
we may expect that the converse statement would hold.

\begin{conjecture}
\label{conj:bo}
Let $G$ be the fundamental group of a $3$--manifold. 
Then,
$G$ is bi-orderable if and only if
$G$ has no generalized torsion element.
\end{conjecture}

There are several works on the bi-orderability and generalized torsion elements of knot groups.
The knot group of any torus knot is not bi-orderable, 
because it contains generalized torsion elements \cite{NR}.
Thus Conjecture \ref{conj:bo} holds for torus knot groups. 
We remark that the knot exterior of a torus knot is a Seifert fibered manifold.
Other examples are twist knots, which have Conway's notation $[2,2n]$.
The knot group of a twist knot is bi-orderable if $n>0$, not bi-orderable if $n<0$ by \cite{CDN}.
The second named author showed that if $n<0$, then the knot group contains a generalized torsion element \cite{Te}.
This means that Conjecture \ref{conj:bo} holds for twist knot groups as well.  
Torus knot groups and twist knot groups are one-relator groups, 
and \cite[Question 3]{CGW} asks whether the conjecture holds for one-relator knot groups, 
more generally one-relator groups.

We first observe the following, 
which enables us to restrict our attention to fundamental groups of prime $3$--manifolds for Conjecture~\ref{conj:bo}. 

\begin{proposition}
\label{prop:sum}
Let $M$ be the connected sum of two $3$--manifolds $M_1$ and $M_2$. 
Suppose that $G_i = \pi_1(M_i)$ satisfies Conjecture~\ref{conj:bo} for $i=1,2$. 
Then $G = \pi_1(M)$ also satisfies Conjecture~\ref{conj:bo}.  
\end{proposition}

The main purpose of this paper is to confirm 
Conjecture \ref{conj:bo} for the fundamental groups of Seifert fibered manifolds, 
Sol manifolds, 
which are possibly non-orientable.

\begin{theorem}
\label{thm:main}
Let $M$ be a compact connected $3$--manifold, and let $G$ be its fundamental group.
If $M$ is either Seifert fibered or Sol, then
$G$ satisfies Conjecture \ref{conj:bo}.
\end{theorem}

Any closed geometric $3$--manifold which possesses a geometric structure other than a hyperbolic structure 
is Seifert fibered or admits a Sol structure \cite[Theorem~5.1]{S}. 
Thus Theorem~\ref{thm:main} shows: 

\begin{corollary}
\label{geometry}
The fundamental group of any closed, geometric $3$--manifold that is non-hyperbolic satisfies Conjecture~\ref{conj:bo}. 
\end{corollary}

The $n$--fold cyclic branched cover $\Sigma_n$ of the $3$--sphere branched
over the figure-eight knot is known to be an $L$--space and
have non-left-orderable fundamental group \cite{DPT,P,Te0}.
In particular, $\Sigma_n$ is hyperbolic if $n\ge 4$.

\begin{theorem}\label{thm:fe}
Let $\Sigma_n$ be the $n$--fold cyclic branched cover of $S^3$ over the figure-eight knot.
Then $\pi_1(\Sigma_n)$ satisfies Conjecture \ref{conj:bo}.
\end{theorem}

Section \ref{sec:Seifert} treats the case where $M$ is a Seifert fibered manifold,
and Section \ref{sec:sol} examines the case where $M$ is a Sol--manifold. 
Theorem~\ref{thm:main} follows from Theorems~\ref{Seifert_g-torsion} and \ref{Sol_g-torsion}. 
In Section \ref{sec:hyp} we prove that each generator in the standard cyclic presentation of the Fibonacci group 
$F(2,m)\ (m>2)$ is a generalized torsion element (Theorem~\ref{Fibonacci}). 
Since $\pi_1(\Sigma_n)$ is isomorphic to $F(2, 2n)$ \cite{HKM,HLM}, 
this result immediately implies Theorem \ref{thm:fe}. 
We also verify the conjecture for another infinite family of closed hyperbolic $3$--manifolds, 
which are the first ones that do not contain Reebless foliations given by \cite{RSS}.

\section{Preliminaries}

In a group, we use the notation $g^a=a^{-1}ga$ for a conjugate
and $[a,b]=aba^{-1}b^{-1}$ for a commutator. 

We recall some results which will be useful in the proof of Theorem~\ref{thm:main}. 

\begin{lemma}
\label{lem:kb}
Let $K$ be the Klein bottle.
Then $\pi_1(K)$ contains a generalized torsion element.
\end{lemma}

\begin{proof}
It is well known that $\pi_1(K)$ has a presentation
\[
\pi_1(K)=\langle x,y  \mid y^{-1}xy=x^{-1}\rangle.
\]
Since $xx^y=1$ from the relation and $x\ne 1$,
$x$ is a generalized torsion element.
\end{proof}

Lemma 5.1 in \cite{H} shows: 

\begin{lemma}
\label{P2}
If a $3$--manifold $M$ contains a projective plane, 
then $\pi_1(M)$ admits a torsion element, hence a generalized torsion element. 
\end{lemma}

\begin{lemma}
\label{lem:bo}
If $G$ is bi-orderable, then $G$ has no generalized torsion element.
\end{lemma}

\begin{proof}
Let $<$ be bi-ordering of $G$.
Suppose that $G$ contains a generalized torsion element $g$.
Therefore, there exist $a_1,\dots,a_n \in G$ such that
\[
g^{a_1}g^{a_2}\dots g^{a_n}=1.
\]

Since $g\ne 1$,  we have $g>1$ or $g<1$.
If $g>1$, then $g^{a_i}>1$  for any $i$ by bi-orderability.
So, the product of these conjugates is still bigger than $1$,
a contradiction.  The case $g<1$ is similar.
\end{proof}

We recall the following result due to Vinogradov \cite{V}. 

\begin{lemma}
\label{free_product}
A free product $G = G_1 * G_2 * \cdots * G_n$ of groups is bi-orderable 
if and only if each $G_i$ is bi-orderable. 
\end{lemma}

\begin{proof}[Proof of Proposition~\ref{prop:sum}]
If $G$ is bi-orderable, 
then $G$ has no generalized torsion element (Lemma \ref{lem:bo}). 
Conversely, assume that $G$ is not bi-orderable. 
Then it follows from Lemma~\ref{free_product} that $G_1$ or $G_2$ is not bi-orderable. 
Without loss of generality, we may assume $G_1$ is not bi-orderable. 
By the assumption $G_1$ has a generalized torsion element, 
which is also a generalized torsion element of $G$. 
\end{proof}

\section{Seifert fibered manifolds}
\label{sec:Seifert}

The goal in this section is to establish Conjecture~\ref{conj:bo} for Seifert fibered manifolds, 
which may be non-orientable.  
Since any bi-orderable group has no generalized torsion element (Lemma~\ref{lem:bo}), 
it is sufficient to show the following. 

\begin{theorem}
\label{Seifert_g-torsion}
Let $M$ be a Seifert fibered manifold which is possibly non-orientable.  
If $G= \pi_1(M)$ is not bi-orderable, 
then $G$ has a generalized torsion element.
 \end{theorem}

Before proving the theorem, 
we recall the characterization of Seifert fibered manifolds whose fundamental groups are bi-orderable 
due to Boyer, Rolfsen and Wiest \cite{BRW}. 

\begin{theorem}[\cite{BRW}]
\label{thm:BRW}
Let $M$ be a compact connected Seifert fibered manifold, 
and let $G$ be its fundamental group.
Then $G$ is bi-orderable if and only if either
\begin{enumerate}
\item  $G$ is the trivial group and $M=S^3$\textup{;} or
\item  $G$ is infinite cyclic and $M$ is either 
$S^1\times S^2$, $S^1\tilde{\times}S^2$ or a solid Klein bottle\textup{;} or
\item $M$ is the total space of a locally trivial, orientable circle bundle over
a surface other than $S^2$, $P^2$ or the Klein bottle.
\end{enumerate}
\end{theorem}

We should remark that in case (3) of Theorem \ref{thm:BRW},
$M$ is not necessarily orientable.
A circle bundle over a surface is said to be \textit{orientable\/}
if for any loop on the base surface, its preimage under the natural projection 
is a torus. 
So, the total space of an orientable circle bundle may be non-orientable.  
In case (3), $M$ is a non-orientable $3$--manifold,
whenever the base surface is non-orientable.
For example, the trivial circle bundle over the M\"{o}bius band is
a non-orientable Seifert fibered manifold,
and its fundamental group is $\mathbb{Z}^2$, which is bi-orderable.

Based on the characterization in Theorem \ref{thm:BRW},
we will show that if the fundamental group of a Seifert fibered manifold $M$
is not bi-orderable, then it contains a generalized torsion element.
The proof of Theorem \ref{Seifert_g-torsion} is divided into two cases according as $M$ is orientable or not.
The two cases are discussed in Subsections \ref{subsec:ori} and \ref{subsec:nonori}, respectively.

Let $M$ be a compact connected Seifert fibered manifold,
and $G$ the fundamental group of $M$.
Suppose that $G$ is not bi-orderable hereafter.

\subsection{Proof of Theorem~\ref{Seifert_g-torsion} for orientable Seifert fibered manifolds}
\label{subsec:ori}

In this section, 
we assume that $M$ is an orientable Seifert fibered manifold whose fundamental group $G$ 
is not bi-orderable. 
We will look for a generalized torsion element in $G$.

First, we make a reduction.
Since the trivial group is bi-orderable, $G$ is non-trivial.
If $M$ is reducible, then $M$ is either $S^1\times S^2$ or $P^3\# P^3$.
For the first case, $G$ is infinite cyclic, so bi-orderable.
In the second case $G = \mathbb{Z}_2 * \mathbb{Z}_2$ has a torsion element.  
Thus in the following we assume that $M$ is irreducible.

Fix a Seifert fibration $\mathcal{F}$ of $M$, 
and let $B$ be a base surface obtained by identifying each fiber to a point. 
Then we have a natural projection $p : M \to B$. 
The Seifert fibration $\mathcal{F}$ gives $B$ an orbifold structure, 
and we denote the base orbifold by $\mathcal{B}$.

The case where $B$ is non-orientable is easy to settle.

\begin{lemma}\label{lem:nonori}
If $M$ is orientable and $B$ is non-orientable, then $G$ contains a generalized torsion element.
\end{lemma}

\begin{proof}
Let $\ell$ be an orientation-reversing loop on $B$.
Then the inverse image $p^{-1}(\ell)$ gives the Klein bottle  $K$ in $M$.
Let $T$ be the torus boundary of the regular neighborhood $N(K)$ of $K$, which
is the twisted $I$-bundle over the Klein bottle.
By Lemma \ref{lem:kb},
$\pi_1(N(K))\ (=\pi_1(K))$ contains a generalized torsion element.

If the torus $T$ is incompressible in $M$, then 
$\pi_1(N(K))$ is a subgroup of $G$.
Hence the above generalized torsion element remains in $G$.

If $T$ is compressible, then $T$ bounds a solid torus by the irreducibility of $M$.
Hence $M$ is the union of the twisted $I$--bundle over the Klein bottle and a solid torus.
Then $M$ is either $S^1\times S^2$, $P^3\# P^3$, a lens space or a prism manifold.
The first case is eliminated by our assumption that $G$ is not bi-orderable.
When the second case happens, 
$P^3\# P^3$ is reducible, contradicting the assumption. 
For the remaining cases, $G$ is finite, so it contains a torsion element.
\end{proof}

Let $n$ be the number of exceptional fibers in $\mathcal{F}$.

\begin{lemma}\label{lem:n0}
If $n=0$, then $G$ contains a generalized torsion element.
\end{lemma}

\begin{proof}
By Lemma \ref{lem:nonori}, we can assume that $B$ is orientable.
Since $M$ is a circle bundle over $B$,
$B$ is $S^2$ by Theorem \ref{thm:BRW}. 
Then $M$ is $S^3$, $S^1\times S^2$ or a lens space.
Since $G$ is not bi-orderable, $M$ is a lens space.
Hence $G$ contains a torsion.
\end{proof}

\begin{lemma}\label{lem:efiber}
If $G$ is infinite and non-abelian, and $n>0$, then $G$ contains a generalized torsion element.
\end{lemma}

\begin{proof}
Again, we can assume that $B$ is orientable by Lemma \ref{lem:nonori}.
(Then the canonical subgroup in the sense of \cite{JS} coincides with $G$.)
Let $e$ be the element represented by an exceptional fiber of index $\alpha\ (\ge 2)$.
By \cite[II.4.7]{JS} (which needs the assumption that $G$ is infinite),
the centralizer of $e$ is abelian, because $e$ does not lie in
the subgraph generated by a regular fiber $h$, which is
infinite cyclic and normal.

Thus the centralizer of $e$ is strictly smaller than $G$.
Hence there exists an element $f\in G$ which does not commute with $e$.
However, $e^{\alpha}=h$, the element represented by a regular fiber, so $e^{\alpha}$ is central in $G$.
Thus the commutator $[e,f]\ne 1$, but
$[e^{\alpha},f]=1$.
We remark that $[e^{\alpha},f]$ is a product of conjugates of $[e,f]$,
which follows inductively from the equation
\[
[e^{\alpha},f]=[e^{\alpha-1},f]^{e^{-1}}[e,f].
\]
This implies that the commutator $[e,f]$ is a generalized torsion element. 
\end{proof}

It follows from Lemmas \ref{lem:nonori} and \ref{lem:n0} that 
we can assume that $B$ is orientable and $n>0$. 
We now separate into two cases depending upon $\partial B = \varnothing$ or not. 

\medskip

\textbf{Case 1.} $\partial B=\varnothing$.

Let $g$ be the genus of the closed orientable surface $B$.
If $g=0$ and $n\le 2$, then $M$ is $S^3$, $S^1\times S^2$ or a lens space.
Since $G$ is not bi-orderable, $M$ is a lens space.
Then, $G$ contains a torsion.

Suppose $g=0$ and $n\ge 3$, or $g\ge 1$.

We claim that $G$ is non-abelian.
If $G$ is abelian, then $M$ is either
$S^1\times S^2$, $T^3$, or a lens space; 
see \cite[p.25]{AFW}.
For the first two case, $G$ is bi-orderable.
Hence $M$ is a lens space, but this is impossible by the 
assumption $g=0$ and $n\ge 3$, or $g\ge 1$.

If $G$ is finite, then $G$ contains a torsion.
Otherwise, the conclusion follows from Lemma \ref{lem:efiber}.

\medskip

\textbf{Case 2.} $\partial B\ne \varnothing$.

If $B$ is the disk with $n=1$, then $M$ is a solid torus.
Then $G$ is infinite cycle, which is bi-orderable.

If $B$ is either the disk with $n=2$,
or an annulus with $n=1$,
then Lemma \ref{lem:efiber} gives the conclusion.

Except these three cases,
we can choose a loop $\ell$ on $B$ such that either 
\begin{enumerate}
\item $\ell$ bounds a disk with two cone points (of $\mathcal{B}$); or
\item $\ell$ and one boundary component of $B$ cobounds an annulus with one cone point (of $\mathcal{B}$),
\end{enumerate}
and that the inverse image $p^{-1}(\ell)$ under the natural projection $p : M \to B$
gives a separating incompressible torus  $T$ in $M$.

Then the fundamental group of one side of $T$ in $M$ contains
a generalized torsion as above, which remains in $G$. 
This completes the proof of Theorem~\ref{Seifert_g-torsion} for orientable Seifert fibered manifolds.  

\subsection{Proof of Theorem~\ref{Seifert_g-torsion} for non-orientable Seifert fibered manifolds}
\label{subsec:nonori}

In this section, we examine a non-orientable Seifert fibered manifold $M$
with fundamental group $G$.
Let $n$ denote the number of (isolated) exceptional fibers, 
which are orientation-preserving in $M$.
Exceptional fibers which are orientation-reversing, if they exist,
form one-sided annuli, tori or Klein bottles in $M$ \cite[p.431]{S}.
After \cite{O},
we call such exceptional fibers \textit{special exceptional fibers}.

Recall that we assume that $G$ is not bi-orderable.
Our goal is to find a generalized torsion element in $G$.  
 
\begin{lemma}\label{lem:lift}
If $n>0$, then $M$ contains a generalized torsion element.
\end{lemma}

\begin{proof}
Assume $n>0$.
Take an orientation cover $\tilde{M}$ of $M$.
It is the unique double cover of $M$, which corresponds
to the kernel of the surjection from $G$ to $\mathbb{Z}_2$,
sending the element of $G$ to $0$ or $1$ according as
the loop is orientation-preserving or not.
Also, the Seifert fibration of $M$ naturally lifts to one of $\tilde{M}$.

Let $e$ be an isolated exceptional fiber in $M$.
Since $e$ is orientation-preserving,
it lifts to an isolated exceptional fiber of $\tilde{M}$ with the same index.

If $\pi_1(\tilde{M})$ is not bi-orderable,
then it contains a generalized torsion element by 
the orientable case of Theorem \ref{Seifert_g-torsion}, 
which is established in Section \ref{subsec:ori}.
Since $\pi_1(\tilde{M})$ is a subgroup of $G$,
the generalized torsion element remains in $G$.
Therefore, we now assume that $\pi_1(\tilde{M})$ is bi-orderable, though
$\pi_1(M)$ is not bi-orderable.
Then, by Theorem \ref{thm:BRW},
there are three possibilities for $\tilde{M}$ which is orientable.

\medskip

\textbf{Case 1.}
$\tilde{M}$ is $S^3$.

In this case, $M$ is the quotient of $S^3$ under $\mathbb{Z}_2$--action.
Then $M$ would be orientable (indeed, a lens space), a contradiction; 
see \cite[p.456]{S}. 

\medskip

\textbf{Case 2.}
$\tilde{M}$ is $S^1\times S^2$.

Since $M$ is the quotient of $S^1\times S^2$ under
$\mathbb{Z}_2$--action,
$M$ is either $S^1\times S^2$, $S^1\tilde{\times}S^2$, $P^3\# P^3$,
or $S^1\times P^2$ \cite[p.457]{S}. 
Since $M$ is non-orientable, 
$M$ is either $S^1\tilde{\times}S^2$ or $S^1\times P^2$. 
In the former, 
$\pi_1(M) = \mathbb{Z}$ is bi-orderable, contradicting the assumption. 
In the latter, 
by Lemma~\ref{P2} $\pi_1(M)$ contains a torsion element, hence a generalized torsion. 

\medskip

\textbf{Case 3.}
$\tilde{M}$ is the total space of a locally trivial, orientable circle bundle over
a surface $\tilde{B}$ other than $S^2$, $P^2$ or the Klein bottle.

Since $\tilde{M}$ is orientable, $\tilde{B}$ is also orientable.
Recall that $\tilde{M}$ has an exceptional fiber in the Seifert fibration coming from $M$.
Hence, if the fibration of $\tilde{M}$ is unique, then this is a contradiction.
From the classification of Seifert fibered manifolds with non-unique fibrations \cite{J},
the only possibility of $\tilde{M}$ is $S^1\times D^2$.
Then $M$ is a fibered solid Klein bottle \cite[p.443]{S}, 
which contradicts the assumption that $G$ is not bi-orderable.
\end{proof}

\begin{lemma}\label{lem:se}
If $M$ contains no exceptional fibers, then $G$
contains a generalized torsion element.
\end{lemma}

\begin{proof}
Since there is no exceptional fiber, $M$ is a circle bundle over a surface $B$.

If $B$ is orientable,
then there exists a loop $\ell$ in $B$
over which fibers cannot be coherently oriented, 
because $M$ is non-orientable.
Then the inverse image $p^{-1}(\ell)$ under the natural projection $p : M\to B$
gives the Klein bottle in $M$.
If $\gamma\in G$ is represented by $\ell$,
then $h^{-1}=\gamma^{-1}h\gamma$,
so $hh^\gamma=1$,
where $h$ is represented by a regular fiber.
We remark that $h\ne 1$ \cite[Proposition 4.1]{BRW}.
Hence $h$ is a generalized torsion element.

Assume now that $B$ is non-orientable.
If there exists a loop in $B$ over which
fibers cannot be coherently oriented, then the above argument works again.
Hence $M$ is an orientable circle bundle over $B$.
By Theorem \ref{thm:BRW},
$B$ must be either $P^2$ or the Klein bottle. 

When $B=P^2$, there are only two orientable circle bundles over $B$,
$S^1 \times P^2 $ and $S^1\tilde{\times}S^2$ \cite[p.279]{BRW}.  
If $M = S^1 \times P^2$, then $G$ has a torsion element, hence a generalized torsion element (Lemma~\ref{P2}). 
If $M = S^1\tilde{\times}S^2$, then $G$ is bi-orderable, contradicting our initial assumption. 

When $B$ is the Klein bottle $K$, there are also two possibilities for $M$, $S^1  \times K$
and the non-trivial circle bundle over $K$.
For the former, $\pi_1(K)$ is a subgroup of $G$.
Since $\pi_1(K)$ contains a generalized torsion element by Lemma \ref{lem:kb}, so does $G$.
For the latter, $G$ has a presentation
\[
G=\langle x, y, h \mid [h,x]=[h,y]=1, x^2y^2=h\rangle
=\langle x,y \mid x^2y^2\ \text{is central}\rangle,
\]
as described in \cite[p.279]{BRW}.
Then
\[
[x^2,y]=x^{2}yx^{-2}y^{-1}=(x^{2}y^2)y^{-1}x^{-2}y^{-1}=y^{-1}x^{-2}(x^2y^2)y^{-1}=1.
\]
Note $[x^2,y]=[x,y]^{x^{-1}}[x,y]$.
Since there is a surjection from $G$ onto the non-abelian group $\langle x,y\mid x^2=y^2=1\rangle
=\mathbb{Z}_2*\mathbb{Z}_2$,
$G$ is not abelian. Hence $[x,y]\ne 1$ in $G$.
Thus $[x,y]$ is a generalized torsion element.
\end{proof}

It follows from Lemmas \ref{lem:lift} and \ref{lem:se} that 
we may assume that $M$ contains a special exceptional fiber $e$.
Then $e^2=h$, which is a regular fiber.

Now, the base surface $B$ has non-empty boundary which contains reflector lines.
Let $N$ be a regular neighborhood of the set of reflector lines in $B$,
and let $N_0$ be a component of $N$.
Decompose $B$ into $N_0$ and $B_0=\textrm{cl}(B-N_0)$.
Then $N_0 \cap B_0$ is either an arc or a circle.
If we put $P_0=p^{-1}(N_0)$ and $M_0=p^{-1}(B_0)$,
then $M$ is decomposed into $P_0$ and $M_0$ along
a vertical annulus or torus, according as 
$N_0 \cap B_0$ is either an arc or a circle.
(A vertical Klein bottle does not appear,  because of the argument
in the second paragraph of the proof  of Lemma \ref{lem:se}.)
In the former case, $P_0$ is a fibered solid Klein bottle, and in the latter case, 
$P_0$ is the twisted $I$--bundle over a torus \cite[pp.433-434]{S}.
In either case, $P_0\cap M_0$ is incompressible in $P_0$.

If $P_0\cap M_0$ is compressible in $M_0$,
then $P_0$ is the twisted $I$--bundle over the torus and
$M_0$ is a solid torus \cite[p.280]{BRW}. 
This implies that $M$ is obtained by Dehn filling on $P_0$, so
its fundamental group $G$ is a quotient of $\mathbb{Z}\oplus \mathbb{Z}$.
Thus $G$ is abelian.
If it is torsion-free, then it is bi-orderable, a contradiction.
Hence $G$ has a (non-trivial) torsion, which is a generalized torsion element.

Finally, we assume that $P_0\cap M_0$ is incompressible in $M_0$.
Then $G$ is the amalgamated free product of $\pi_1(P_0)$ and $\pi_1(M_0)$
over $\pi_1(P_0\cap M_0)$.
It is well known that any element in $\pi_1(P_0)-\pi_1(P_0\cap M_0)$ does not
commute with any element in $\pi_1(M_0)-\pi_1(P_0\cap M_0)$ \cite{MKS}.

If the inclusion $\pi_1(P_0\cap M_0)\to \pi_1(M_0)$ is an isomorphism,
then 
$M_0$ would be the trivial $I$--bundle over an annulus or a torus \cite[Theorems 5.2 and 10.6]{H}. 
Then $M$ is homeomorphic to $P_0$, so
$G$ is bi-orderable, a contradiction.
Hence the inclusion $\pi_1(P_0\cap M_0)\to \pi_1(M_0)$ is not an isomorphism. 

We remark that the special exceptional fiber $e$ lies in $\pi_1(P_0)-\pi_1(P_0\cap M_0)$. 
Suppose that there exists an element $f \in \pi_1(M_0)-\pi_1(P_0\cap M_0)$ which commutes with $h$.
Then we have $[e,f]\ne 1$, but $[e^2,f]=[h,f]=1$. 
Since $[e,f]^{e^{-1}}[e,f] = [e^2,f] = 1$, 
$[e,f]$ gives a generalized torsion element in $G$.
So in the following we look for such an element $f \in \pi_1(M_0)-\pi_1(P_0\cap M_0)$. 

If $M_0$ contains a special exceptional fiber,
then it gives the desired element $f$.
Otherwise, $B_0$ does not contain  reflector curves.
If $B_0$ is a disk, 
then $M_0$ is a solid torus and $P_0 \cap M_0$ is an annulus.  
Since the inclusion $\pi_1(P_0\cap M_0) \to \pi_1(M_0)$ is injective, but not surjective, 
the core of the vertical annulus $P_0\cap M_0$ (a regular fiber) intersects a meridian disk of $M_0$ more than once. 
This means that the core of $M_0$ is an exceptional fiber. 
Then we have a generalized torsion element by Lemma \ref{lem:lift}. 
Hence $B_0$ is not a disk, and we take a homotopically nontrivial loop $f$ on $B_0$. 
As before,  if the regular fibers over $f$ cannot be oriented coherently,
then there is the Klein bottle whose fundamental group contains a generalized torsion element. 
Otherwise, $f$ gives the desired element commuting with $h$.
We have thus established Theorem~\ref{Seifert_g-torsion} for non-orientable Seifert fibered manifolds. 

\section{Sol manifolds}\label{sec:sol}

In this section we will prove: 

\begin{theorem}
\label{Sol_g-torsion}
Let $M$ be a Sol manifold. 
If $G= \pi_1(M)$ is not bi-orderable, 
then $G$ has a generalized torsion element.
\end{theorem}

It was shown in \cite{KK, MR0,MR} that if a solvable group with finite rank 
(i.e. there is a universal bound for the rank of finitely generated subgroups) 
has no generalized torsion element, 
then it is bi-orderable.  
Since a Sol manifold has a solvable fundamental group with finite rank \cite{AFW,B},  
the contraposition of Theorem~\ref{Sol_g-torsion}, hence Theorem~\ref{Sol_g-torsion},  
holds.  
However, we give an alternative proof by explicitly identifying a generalized torsion element in $G$. 

The characterization of Sol manifolds with bi-orderable fundamental groups
is also known by \cite{BRW}.

\begin{theorem}[\cite{BRW}]
\label{thm:BRW-sol}
Let $M$ be a compact connected Sol $3$--manifold with fundamental group $G$.
Then $G$ is bi-orderable if and only if either
\begin{itemize}
\item[(1)] $\partial M\ne \varnothing$ and $M$ is not the twisted $I$--bundle
over the Klein bottle\textup{;} or
\item[(2)]  $M$ is a torus bundle over the circle whose monodromy in $GL_2(\mathbb{Z})$ has at least one positive eigenvalue.
\end{itemize}
\end{theorem}

Note that there are two twisted $I$--bundles over the Klein bottle; 
one is orientable and the other is non-orientable \cite{GHG}.

\begin{proof}[Proof of Theorem~\ref{Sol_g-torsion}]
Recall that $M$ is a Sol manifold whose fundamental group $G$ is not bi-orderable. 
In the following we look for a generalized torsion element in $G$. 

\begin{lemma}
\label{lem:sol1}
If $\partial M\ne \varnothing$, then $G$ contains a generalized torsion element.
\end{lemma}

\begin{proof}
Since $G$ is assumed to be not bi-orderable and $\partial M \ne \varnothing$, 
by Theorem~\ref{thm:BRW-sol}, 
$M$ is the twisted $I$--bundle over the Klein bottle. 
Then Lemma \ref{lem:kb} shows that $G$ contains a generalized torsion element.
\end{proof}

Thus we assume that $M$ is closed.
Following \cite[p.282]{BRW}, there are three possibilities for $M$;
\begin{itemize}
\item[(1)] a torus or Klein bottle bundle over the circle; or
\item[(2)] non-orientable and the union of two twisted $I$--bundles over the Klein bottle
which are glued along their Klein bottle boundaries; or
\item[(3)] orientable and the union of two twisted $I$--bundles over the Klein bottle
which are glued along their torus boundaries.
\end{itemize}

Except the case where $M$ is a torus bundle over the circle,
there is a $\pi_1$--injective Klein bottle in $M$.
By Lemma \ref{lem:kb}, $G$ contains a generalized torsion element.
Thus we may assume that $M$ is a torus bundle over the circle with Anosov monodromy $A\in GL_2(\mathbb{Z})$.
By Theorem \ref{thm:BRW-sol} and our assumption that $G$ is not bi-orderable,
$A$ has no positive eigenvalue.
(We remark that $A$ has distinct two real eigenvalues \cite[p.470]{S}.)
Hence the two eigenvalues of $A$ are negative real numbers, 
so $\det A=1$ and $\mathrm{tr}(A)<-2$. 
Theorem~\ref{Sol_g-torsion} now follows from Theorem~\ref{thm:trace} below. 
\end{proof}

For a torus bundle over the circle, 
we can find a generalized torsion element explicitly in its fundamental group
under a weaker condition.

\begin{theorem}
\label{thm:trace}
Let $M$ be a torus bundle over the circle with monodromy $A\in SL_2(\mathbb{Z})$.
If $\mathrm{tr}(A)<0$, then $\pi_1(M)$ contains a generalized torsion
element.
\end{theorem}

\begin{proof}
Let $A=\begin{pmatrix} a & c \\ b & d \end{pmatrix}$, with $ad-bc=1$
and $a+d<0$.
Then we may assume that either $a,d\le 0$, or $a>0$ and $d<0$.

Now, $\pi_1(M)$ has a presentation
\begin{equation}\label{eq:pi1}
\pi_1(M)=\langle l,m,t\mid [l,m]=1, t^{-1}lt=l^am^b, t^{-1}mt=l^cm^d\rangle.
\end{equation}
We will show that the element $l$ is a generalized torsion element.

Since any torus fiber is $\pi_1$-injective, $l\ne 1$.
From the relations, we have
\begin{equation}\label{eq:negative}
(l^t)^{-d}=l^{-ad}m^{-bd}, (m^t)^b=l^{bc}m^{bd}.
\end{equation}
From the 1st one, 
\[
l(l^t)^{-d}=l^{1-ad}m^{-bd}.
\]
Multiplying this with the 2nd relation of (\ref{eq:negative}), and using $ad-bc=1$,
\begin{equation}\label{eq:1}
l(l^t)^{-d}(m^t)^b=1.
\end{equation}

\medskip

\textbf{Case 1.}
$a,d\le 0$

The 2nd relation of (\ref{eq:pi1}) gives
\[
m^b=l^{-a}l^t.
\]
From this and (\ref{eq:1}), we have
\[
l(l^t)^{-d}(l^{-a}l^t)^t=1.
\]
Since the left hand side is a product of the conjugates of $l$,
this shows that $l$ is a generalized torsion element.

\medskip

\textbf{Case 2.}
$a>0$ and $d<0$.

(\ref{eq:1}) is changed to
\[
l(l^{-d}m^b)^t=1.
\]
But
\[
l(l^{-d}m^b)^t=l(l^{-a-d}l^am^b)^t=l(l^{-a-d})^t(l^am^b)^t.
\]
From (\ref{eq:pi1}), $l^t=l^am^b$.
Hence
\[
l(l^{-a-d})^tl^{t^2}=1.
\]
Since $a+d<0$, the left hand side is a product of conjugates of $l$.

Thus we have shown that $l$ is a generalized torsion element. 
\end{proof}

\section{Hyperbolic manifolds}\label{sec:hyp}

Corollary~\ref{geometry} says that Conjecture~\ref{conj:bo} holds for any closed $3$--manifold 
which possesses a geometric structure other than non-hyperbolic structure. 
In this section, 
we first prove Theorem \ref{thm:fe},
and then
we verify the conjecture for some closed hyperbolic $3$--manifolds introduced by 
Roberts, Shareshian and Stein \cite{RSS}.

\subsection{Cyclic branched covers of the figure-eight knot}

Let $K$ be the figure-eight knot, 
and let $\Sigma_n=\Sigma_n(K)$ be the $n$--fold cyclic branched cover of the $3$--sphere $S^3$ branched over $K$.
It is known that $\Sigma_2$ is a lens space, $\Sigma_3$ is Seifert fibered, 
and $\Sigma_n$ is hyperbolic if $n>3$; see \cite{HKM,HLM}. 
Furthermore, any $\Sigma_n$ is an $L$--space \cite{P, Te0},
and has non-left-orderable fundamental group \cite{DPT}.
(A left-ordering in a group $G$ is a strict total ordering
which is invariant under left-multiplication.)
In particular, $\pi_1(\Sigma_n)$ is not bi-orderable.
We prove that the fundamental group of $\Sigma_n$ contains a generalized torsion element
when $n>1$, from which Theorem \ref{thm:fe} immediately follows.

\begin{theorem}
\label{thm:gtorsion-fe}
The fundamental group $G = \pi_1(\Sigma_n)$ contains a generalized torsion element whenever $n > 1$.
\end{theorem}

\begin{proof}
The Fibonacci group $F(2,m)$,
introduced by Conway \cite{Con}, 
has a presentation: 
\[
F(2,m)=\langle a_1,a_2,\dots,a_{m} \mid a_i a_{i+1}=a_{i+2}\ \text{(indices modulo $m$)}\rangle. 
\]
By \cite{HKM,HLM}, 
$G$ is isomorphic to the Fibonacci group $F(2,2n)$.
Theorem~\ref{thm:gtorsion-fe} now follows from Theorem~\ref{Fibonacci} below, 
which we prove a stronger statement for all Fibonacci groups.
\end{proof}

Recall that $F(2, m)$ is a trivial group if $m = 1, 2$. 
When $m > 2$, we establish: 

\begin{theorem}
\label{Fibonacci}
In the Fibonacci group $F(2,m)\ (m > 2)$, 
each generator $a_i$ is a generalized torsion element. 
\end{theorem}

\begin{proof}
It is sufficient to show that $a_1$ is a generalized torsion element.
From the presentation, 
it is easy to see that $F(2,m)$ is generated by $a_1$ and $a_2$.
For simplicity, let $a=a_1$ and $b=a_2$.
 
\begin{claim}\label{cl:ane1}
$a\ne 1$ in $F(2,m)$.
\end{claim}
 
\begin{proof}
Assume for a contradiction that $a = a_1 = 1$ in $F(2, m)$. 
Then a repeated application of relations shows that $F(2, m)$ is generated by a single element $a_2$,
and moreover, $F(2,m)$ would be finite cyclic.
By a direct calculation, we have that 
$F(2,3)=\mathbb{Z}_2$, $F(2,4)=F(2,5)=F(2,7)=\{1\}$.

On the other hand,
it is known that $F(2,m)$ is finite if and only if $m=1, 2, 3, 4, 5, 7$, 
and that $F(2,3)$ is the quaternion group,
$F(2,4)=\mathbb{Z}_5$, $F(2,5)=\mathbb{Z}_{11}$, $F(2,7)=\mathbb{Z}_{29}$ \cite{JWW,N}. 
We have a contradiction.   

\end{proof}

From the relations, $a_3=a_1a_2=ab$, $a_4=a_2a_3=bab$.
Thus we have the expressions recursively
\[
a_3=ab, a_4=bab,\ a_5=ab^2ab,\ a_6=babab^2ab, \dots.
\]
We call these the \textit{canonical expressions\/} of $a_i$'s $(3\le i\le m)$. 
In the canonical expression of $a_i$, 
neither $a^{-1}$ nor $b^{-1}$ appears. 
Let $e_i$ denote the total exponent sum of $b$ in the canonical expression of $a_i$.
For example, $e_3=1$, $e_4=2$.
From the relation $a_ia_{i+1}=a_{i+2}$,
it is obvious that $e_i=F_{i-1}$, 
which is the $(i-1)$-th Fibonacci number with $F_1=F_2=1$.

Hence, if we rewrite the right hand side of the equation $a_1=a_{m-1}a_{m}$
into the canonical expression,
then the total exponent sum of $b$ in the expression is 
\[
e_{m-1}+e_{m}=F_{m-2}+F_{m-1}=F_{m}.
\]
We express this equation as $a = u(a,b)$, 
where the word $u(a,b)$ contains only $a$ and $b$, 
and the total exponent sum of $b$ in $u(a,b)$ is $F_m$.
Furthermore, take the inverse of both sides.
Then we have the equation $a^{-1}= \overline{u}(a^{-1},b^{-1})$, 
where the word $\overline{u}(a^{-1},b^{-1})$ contains only $a^{-1}$ and $b^{-1}$, 
and the total exponent sum of $b$ in $\overline{u}(a^{-1},b^{-1})$ is $-F_{m}$.

On the other hand, 
the relation $a_{m} a_1 = a_2$ enables us to express 
$a_{m}=a_2a_1^{-1}=ba^{-1}$. 
Similarly, we have $a_{m-1}=a_1a_{m}^{-1}=a^2b^{-1}$
from the relations. 
Thus each $a_i$ has yet another expression: 
\[
a_{m}=ba^{-1},\ a_{m-1}=a^2b^{-1},\ a_{m-2}=ba^{-1}ba^{-2},\ 
a_{m-3}=a^2b^{-1}a^2b^{-1}ab^{-1},\dots.
\]
These are called the \textit{non-canonical expressions\/} of $a_i$'s
$(3\le i\le m)$.

Denote by $\overline{e}_i$ the total exponent sum of $b$ in the non-canonical expression of $a_i$.
For example, $\overline{e}_{m}=1$, $\overline{e}_{m-1}=-1$. 
Then it is easy to see that $\overline{e}_i=(-1)^{m+i} F_{m+1-i}$.
Moreover,
in the non-canonical expression of $a_i$,
neither $a$ nor $b^{-1}$ appears when $i=m,m-2,\dots$,
and neither $a^{-1}$ or $b$ appears when $i=m-1,m-3,\dots$.
Also, if $i=m-1,m-3,\dots$, the first letter of the non-canonical expression of $a_i$
is $a$, and the total exponent sum of $a$ is at least two.

As we mentioned above, 
each $a_i$ ($3 \le i \le m$) has the non-canonical expression. 
Using the relations $a_2 = a_4 a_3^{-1}$ and $a_1 = a_3a_2^{-1}$, 
we naturally extend non-canonical expressions to $a_1$ and $a_2$ so that 
$\overline{e}_2 = (-1)^{m+2}F_{m-1}$ and $\overline{e}_1 = (-1)^{m+1}F_{m}$. 
Then rewrite the right hand side of $a = a_1$ into the non-canonical expression
to obtain $a = w_e(a, b^{-1})$ if $m$ is even, 
$w_o(a^{-1},b)$ if $m$ is odd, 
where $w_e(a, b^{-1})$ or $w_o(a^{-1},b)$ is the non-canonical expression of $a_1$ respectively. 
Note also that  $w_e(a,b^{-1})$ contains neither $a^{-1}$ nor $b$,
and $w_o(a^{-1},b)$ contains neither $a$ nor $b^{-1}$.

Now we are ready to identify a generalized torsion element in $F(2,m)$. 

Assume first that $m$ is even.
Then the first letter of the word $w_e(a,b^{-1})$ is $a$.
By canceling the first letter $a$ from both sides of the equation $a=w_e(a,b^{-1})$,
we obtain a new equation $1=w'_e(a,b^{-1})$, 
where $w_e'(a,b^{-1})$ still contains neither $a^{-1}$ nor $b$. 
Moreover, $w_e'(a,b^{-1})$ contains at least one occurrence of $a$.
Since $\overline{e}_1 = - F_{m}$, 
the total exponent sum of $b$ in $w_e'(a,b^{-1})$ is $-F_{m}$.
If we replace any single occurrence of $a$ in $w_e'(a,b^{-1})$
with $a = u(a,b)$, coming from canonical expressions, 
then we have an equation $1 = w(a,b,b^{-1})$, 
where $w(a,b,b^{-1})$ contains no $a^{-1}$. 
Since the total exponent sum of $b$ in $u(a,b)$ is $F_m$ as mentioned before, 
the total exponent sum of $b$ in $w(a,b,b^{-1})$ is $-F_m + F_m = 0$. 

Let us assume that $m$ is odd. 
The equation $a = w_o(a^{-1},b)$ gives $1=a^{-1}\cdot w_o(a^{-1},b)$.
Then replace the first $a^{-1}$ in the right hand side with the word $\overline{u}(a^{-1},b^{-1})$ coming from
the canonical expressions.
This gives $1=\overline{u}(a^{-1},b^{-1}) \cdot w_o(a^{-1},b)$.
The total exponent sum of $b$ in $\overline{u}(a^{-1},b^{-1})$ is $-F_m$,
and that in $w_o(a^{-1},b)$ is $F_m$. 
If we express the right hand side as $w(a^{-1},b,b^{-1})$, which
contains no $a$, 
then the total exponent sum of $b$ in $w(a^{-1},b,b^{-1})$ is $-F_m + F_m = 0$.  

\begin{claim}
\label{cl:product_conjugate}
The word $w(a, b, b^{-1})$ \textup{(}resp. $w(a^{-1},b,b^{-1})$\textup{)}
can be expressed as the product of conjugates of $a$ \textup{(}resp. $a^{-1}$\textup{)}.
\end{claim}

\begin{proof}
We may write 
$$w(a, b, b^{-1}) = a^{m_1}b^{n_1} a^{m_2} b^{n_2} \cdots a^{m_k} b^{n_k},$$ 
where $m_1 \ge 0, m_i >0$ ($2 \le i \le k$), $n_i \ne 0$ ($i \ne k$) and 
$n_1 + \cdots + n_k = 0$. 
Then we rewrite: 
\begin{eqnarray*}
w(a, b, b^{-1}) &=& a^{m_1}b^{n_1} a^{m_2} b^{n_2} \cdots a^{m_k} b^{n_k}\\
   &=& a^{m_1}(b^{n_1} a^{m_2} b^{-n_1}) b^{n_1} b^{n_2} \cdots a^{m_k} b^{n_k}\\
   &=& a^{m_1}(a^{m_2})^{b^{-n_1}}   b^{n_1+ n_2} a^{m_3} b^{n_3} \cdots a^{m_k} b^{n_k}\\
   &=& a^{m_1}(a^{m_2})^{b^{-n_1}}   b^{n_1+ n_2} a^{m_3} b^{-n_1- n_2}  b^{n_1+ n_2+n_3} \cdots a^{m_k} b^{n_k}\\
   &=& a^{m_1}(a^{m_2})^{b^{-n_1}} (a^{m_3})^{b^{-n_1-n_2}} b^{n_1+ n_2+n_3} \cdots a^{m_k} b^{n_k}\\
   &\vdots & \\
   &=& a^{m_1}(a^{m_2})^{b^{-n_1}} (a^{m_3})^{b^{-n_1-n_2}} \cdots (b^{n_1 + \cdots + n_{k-1}} a^{m_k} b^{n_k})\\
   &=& a^{m_1}(a^{m_2})^{b^{-n_1}} (a^{m_3})^{b^{-n_1-n_2}} \cdots (b^{-n_k} a^{m_k} b^{n_k})\\
   &=& a^{m_1}(a^{m_2})^{b^{-n_1}} (a^{m_3})^{b^{-n_1-n_2}} \cdots (a^{m_k})^{b^{n_k}}\\
   &=& a^{m_1}(a^{b^{-n_1}})^{m_2} (a^{b^{-n_1-n_2}})^{m_3}  \cdots (a^{b^{n_k}})^{m_k}.
\end{eqnarray*}
The proof for the word $w(a^{-1},b,b^{-1})$ is similar.
\end{proof}

If a finite product of conjugates of $a^{-1}$ becomes the identity, 
then, taking its inverse, we have a finite product of conjugates of $a$ which is the identity. 
Thus in either case in Claim~\ref{cl:product_conjugate},
some product of conjugates of $a$ yields the identity.
Since $a\ne 1$ in $F(2, m)$ by Claim \ref{cl:ane1}, 
$a$ is a generalized torsion element. 
This completes the proof of Theorem~\ref{Fibonacci}. 
\end{proof}

\begin{remark}
\label{Fibonacci_odd}
\begin{itemize}
\item[(1)]
As mentioned in the proof of Claim \ref{cl:ane1},
$F(2,m)$ is a non-trivial finite group if $m=3,4,5,7$.
Hence any non-trivial element is a torsion element, so a generalized torsion element. 
Furthermore, 
$F(2,2n+1)$ has a non-trivial torsion \cite[Proposition 3.1]{BV}, 
but $F(2,2n)$ is torsion-free if $n>2$. 
\item[(2)]
$F(2, 2n)$ is the fundamental group of $\Sigma_n$. 
On the contrary,  
recently Howie and Williams \cite[Theorem~2.4]{HW} proved that 
$F(2, 2n+1)$ can be the fundamental group of a $3$--manifold if and only if $n = 1, 2$ or $3$. 
\end{itemize}
\end{remark}


\subsection{Other hyperbolic manifolds}

For integers $p, q, m$ with $\gcd(p,q)=1$,
define 
\begin{equation}\label{eq:rss}
G(p,q,m) = \langle a,b,t \mid t^{-1}at = aba^{m-1}, t^{-1}bt = a^{-1}, t^p[a,b]^q=1\rangle.
\end{equation}

In \cite[Proposition 3.1]{RSS}, it is shown that
if $m < 0$, $p > q \ge 1$, $\gcd(p,q)=1$,
then the image of any homomorphism from $G(p,q,m)$ to $\mathrm{Homeo}^+(\mathbb{R})$ is trivial.
This implies that $G(p,q,m)$ is not left-orderable; 
see \cite[Section 5]{BRW}.  
Hence $G(p,q,m)$ is not bi-orderable.

As shown in \cite{RSS}, $G(p,q,m)$ is the fundamental group
of a closed $3$--manifold $M(p,q,m)$ which is
obtained from a once-puncture torus bundle by Dehn filling.
They show that if $m<-2$ and $p$ are odd, $\gcd(p,q)=1$ and $p\ge q\ge 1$,
then $M(p,q,m)$ is hyperbolic for all except finitely many pairs $(p,q)$ \cite[Theorem A]{RSS}. 

Under a certain condition, 
we can show that $G(p,q,m)$ contains a generalized torsion element.

\begin{theorem}
If $p \ge 2q>1$, then $G(p,q,m)$ contains a generalized torsion element.
\end{theorem}

\begin{proof}
We will prove that the element $t$ is a generalized torsion element.

First, $t\ne 1$, because it goes to a non-trivial element
under the abelianization (we need $p > 1$ here).

The 2nd relation $a^{-1}=t^{-1}bt$ of (\ref{eq:rss}) gives
\[
[a,b]=aba^{-1}b^{-1}=t^{-1}b^{-1}tbt^{-1}btb^{-1}.
\]
It is straightforward to verify that
\[
\begin{split}
[a,b]^q &=(t^{-1}b^{-1}tbt\cdot t^{-2}btb^{-1}t^2)(t^{-3}b^{-1}tbt^3\cdot t^{-4}btb^{-1}t^4)
\dots \\
&\quad (t^{-(2q-1)}b^{-1}tbt^{2q-1}\cdot t^{-2q}btb^{-1}t^{2q})t^{-2q}\\
&= (t^{bt}\cdot t^{b^{-1}t^2})(t^{bt^3}\cdot t^{b^{-1}t^4})\dots
 (t^{bt^{2q-1}}\cdot t^{b^{-1}t^{2q}})t^{-2q}.
\end{split}
\]
Hence, the 3rd relation of (\ref{eq:rss}) gives
\[
t^{p-2q}(t^{bt}\cdot t^{b^{-1}t^2})(t^{bt^3}\cdot t^{b^{-1}t^4})\dots
 (t^{bt^{2q-1}}\cdot t^{b^{-1}t^{2q}})=1.
\]
If $p \ge 2q$, then the left hand side is a product of conjugates of $t$.
Thus we have shown that the element $t$ is a generalized torsion element.
\end{proof}

\bibliographystyle{amsplain}

\end{document}